\documentclass[11pt,english]{amsart}

\usepackage{amsmath,leftidx,amsthm}
\usepackage[all]{xy}
\usepackage{xspace}
\usepackage[psamsfonts]{amssymb}
\usepackage[latin1]{inputenc}
\usepackage{graphicx,color}
\usepackage{hyperref,fancyhdr}
\usepackage{fullpage}

\newtheorem*{question}{\bf Question}

\newtheorem{theorem}{\bf Theorem}[section]

\newtheorem{definition}[theorem]{\bf Definition}

\newtheorem{lemma}[theorem]{\bf Lemma}
\newtheorem{corollary}[theorem]{\bf Corollary}

\theoremstyle{remark}

\newtheorem{example}{\bf Example}

\DeclareMathOperator*\lowlim{\underline{lim}}
\DeclareMathOperator*\uplim{\overline{lim}}

\usepackage{versions}
\usepackage{textcomp,listings}

\definecolor{green}{rgb}{0,0.5,0}
\definecolor{dkgreen}{rgb}{0,0.6,0}
\definecolor{gray}{rgb}{0.5,0.5,0.5}
\definecolor{mauve}{rgb}{0.58,0,0.82}
\lstnewenvironment{python}[1][]{
\lstset{
  frame=single,                   
  language=python,                          
  basicstyle=\ttfamily\small, 
  numbers=left,                             
  numberstyle=\scriptsize\color{black},  
  stepnumber=1,                   
  numbersep=7pt,                  
  backgroundcolor=\color{white},      
  stringstyle=\color{red},
  showspaces=false,               
  showstringspaces=false,         
  showtabs=false,                 
  tabsize=2,                      
  captionpos=b,                   
  breaklines=true,                
  breakatwhitespace=false,        
  rulecolor=\color{black},        
  framexleftmargin=1mm, framextopmargin=1mm, frame=shadowbox, 
  alsoletter={1234567890},
  otherkeywords={\ , \}, \{},
  keywordstyle=\color{blue},       
  stringstyle=\color{mauve},         
  emph={access,and,break,class,continue,def,del,elif ,else,%
  except,exec,finally,for,from,global,if,import,in,i s,%
  lambda,not,or,pass,print,raise,return,try,while},
  emphstyle=\color{black}\bfseries,
  emph={[2]True, False, None, self},
  emphstyle=[2]\color{blue},
  emph={[3]from, import, as},
  emphstyle=[3]\color{blue},
  upquote=true,
  morecomment=[s]{"""}{"""},
  commentstyle=\color{dkgreen},       
  emph={[4]1, 2, 3, 4, 5, 6, 7, 8, 9, 0},
  emphstyle=[4]\color{blue},
  literate=*{:}{{\textcolor{blue}:}}{1}%
  {=}{{\textcolor{blue}=}}{1}%
  {-}{{\textcolor{blue}-}}{1}%
  {+}{{\textcolor{blue}+}}{1}%
  {*}{{\textcolor{blue}*}}{1}%
  {!}{{\textcolor{blue}!}}{1}%
  {(}{{\textcolor{blue}(}}{1}%
  {)}{{\textcolor{blue})}}{1}%
  {[}{{\textcolor{blue}[}}{1}%
  {]}{{\textcolor{blue}]}}{1}%
  {<}{{\textcolor{blue}<}}{1}%
  {>}{{\textcolor{blue}>}}{1},%
}}{}

\excludeversion{code}

\def\and{{\quad\text{and}\quad}}

\title{Geometric limits of Julia sets for sums of power maps and polynomials}

\author{Micah Brame}
\address{Butler University, 4600 Sunset Ave., Indianapolis, IN 46208, USA}
\email{mbrame@butler.edu}
\author{Scott Kaschner}
\address{Butler University, 4600 Sunset Ave., Indianapolis, IN 46208, USA}
\email{skaschne@butler.edu}

\begin{document}

\begin{abstract}
For maps of one complex variable, $f$, given as the sum of a degree $n$ power map and a degree $d$ polynomial, we provide necessary and sufficient conditions that the geometric limit as $n$ approaches infinity of the set of points that remain bounded under iteration by $f$ is the closed unit disk or the unit circle.  We also provide a general description, for many cases, of the limiting set. 
\end{abstract}

\maketitle

\section{Introduction}

Let $q$ be a degree $d$ polynomial; define $f_{n}\colon\mathbb C\rightarrow\mathbb C$ by
\[f_{n}(z)\ =\ z^n+q(z),\]
and note that $f_{n}$ is the sum of a power map (whose power we increase in the limit) and a fixed degree $d$ polynomial, $q$.  
For a map $f\colon\mathbb C\rightarrow\mathbb C$, the filled Julia set for $f$, $K(f)$, is the set of points that remain bounded under iteration by $f$. We use the notation $S_0=\{z\in\mathbb C\colon|z|=1\}$ for the unit circle and $\overline{\mathbb D}=\{z\in\mathbb C\colon|z|\leq1\}$ for the closed unit disk.  The purpose of this study is to describe the limit of $K(f_{n})$ in the Hausdorff topology as $n\rightarrow\infty$.

This work was inspired the 2012 study by Boyd and Schulz \cite{boyd} that included a result for the family $f_{n}$ with $\deg q=0$; that is, $q(z)=c\in\mathbb C$.  Among many other things, they proved
\begin{theorem}[Boyd-Shulz, 2012]
\label{THM:BS} If $q(z)=c$, then under the Hausdorff metric,
\begin{align*}
\mbox{for any $|c|<1$, }&\lim_{n\rightarrow\infty}K(f_{n})=\overline{\mathbb D};\\
\mbox{for any $|c|>1$, }&\lim_{n\rightarrow\infty}K(f_{n})=S_0.
\end{align*}
\end{theorem}
It comes as little surprise that this phenomena is easily disrupted.  It was shown in \cite{krs} that when $q(z)=c$ with $|c|=1$, the limiting behavior of $K(f_{n})$ depends on number-theoretic properties of $c$ and the limit almost always fails to exist.  Another study by Alves \cite{alves}, has shown that for maps of the form $f_{n,c}(z)=z^n+cz^k$ for a fixed positive integer $k$, if $|c|<1$, then the limit of $K(f_{n,c})$ as $n\rightarrow\infty$ is $S_0$.

Returning to the more general case in which $q$ is any polynomial, the limiting behavior of $K(f_n)$ is substantially more interesting.  See Figure \ref{FIG:PICS} for examples of filled Julia sets for $f_{n}$, where $q(z)=z^2+c$ and $|c|<1$, that very clearly fail to limit to either the closed unit disk or the unit circle.  The color gradation in the pictures indicates the number of iterates required to exceed a fixed bound for modulus.

Some results from the $\deg q=0$ cases still hold. If $|z|>1$, we can still expect the image of $z$ under $f_{n}$ to have large modulus for large enough $n$.  Guided by this intuition, we find the following generalization of a lemma from \cite{boyd}. We omit the proof, as it is similar to \cite{boyd}, and adopt the notation
\[\mathbb D_r=\{z\in\mathbb C\colon|z|<r\}\mbox{ and }\overline{\mathbb D}_r=\{z\in\mathbb C\colon|z|\leq r\}.\]
\begin{lemma}
\label{LEM:DISKBOUND}
For any polynomial $q$ and any $\epsilon>0$, there is an $N\geq 2$ such that for all $n\geq N$, 
\[K(f_{n})\subset \mathbb D_{1+\epsilon}.\] 
\end{lemma}

\begin{figure}[h]
\scalebox{.35}{\includegraphics{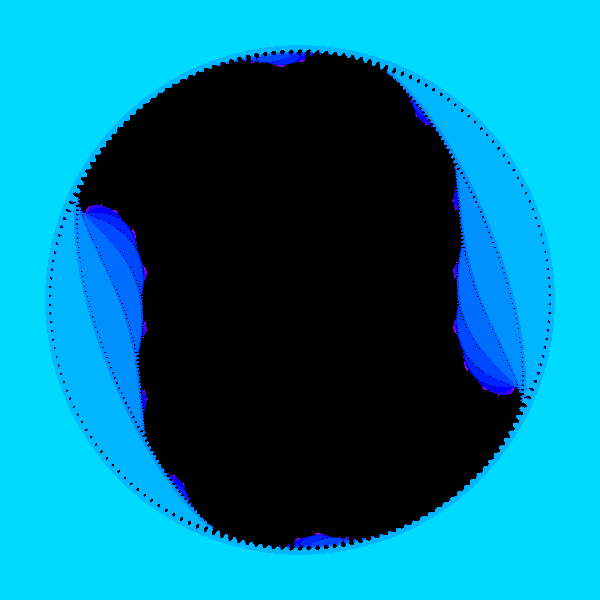}}\hspace{24pt}
\scalebox{.35}{\includegraphics{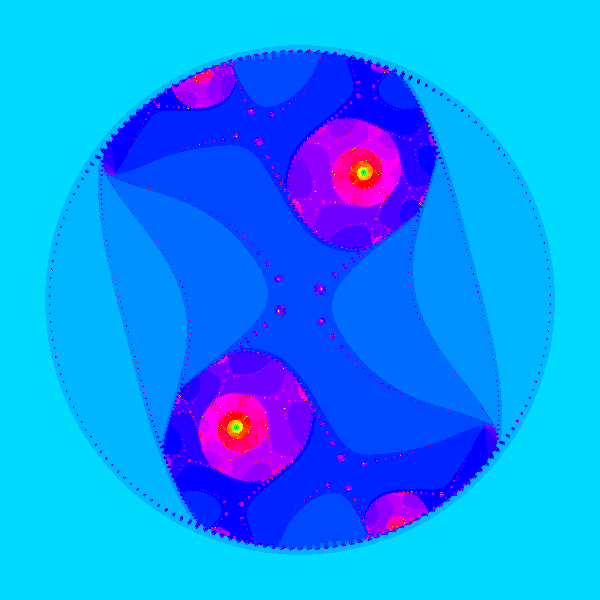}}
\caption{$K(f_{200})$ with $q_1(z)=z^2+0.25+0.25i$ (left) and $q_2(z)=z^2+0.45+0.25i$ (right)}
\label{FIG:PICS}
\end{figure}

Some of the dependence of the limiting behavior of $K(f_{n})$ on $q$ is obvious; by Lemma \ref{LEM:DISKBOUND}, one should expect any point whose orbit by $q$ leaves the unit disk to not be in $K(f_{n})$ for all $n$ sufficiently large.  Thus, one might expect $\uplim_{n\rightarrow\infty}K(f_{n})$ to be contained in the closure of the set
\[\{z \in\mathbb C\colon|q^k(z)|< 1\mbox{ for all }k\}.\]
However, this is not quite the case.  One can prove, as in the $\deg q=0$ cases, that $ S_0$ is always a subset of the $\lowlim_{n\rightarrow\infty}K(f_{n})$. 

Evidence for this fact (proved in Section \ref{SEC:PROOF}) comes by noting that when $n$ is much larger than the degree of $q$, the $n$ fixed points of $f_{n}$ are roughly equidistributed around the unit circle.  This result is connected to the work of Erd\"os, Turan, et al.~\cite{erdos,hughes,iz} on distribution of zeros for sequences of complex polynomials.

By invariance properties of the filled Julia set, one should then expect the preimages of $S_0$ by $q$ (that still have modulus less than or equal to one) to be contained in $\lowlim_{n\rightarrow\infty}K(f_{n})$. The basins of the fixed points accumulating on $S_0$ and their preimages (appearing as small black spots) can be seen in Figure \ref{FIG:PICS}. These ideas and the preceding lemmas lead to the next definition and theorem.

\begin{definition}
Let 
\begin{equation*}
K_{\infty}:=K_q\cup\bigcup_{j\geq0}S_j,
\end{equation*}
where  $K_q:=\{z\in\mathbb C\colon |q^k(z)|<1\mbox{ for all }k\}$, $S_0$ is the unit circle $\{z\in\mathbb C\colon|z|=1\}$, and for any integer $j\geq0$,
\[S_{j}:=\{z\in\mathbb C\colon |q^j(z)|=1\mbox{ and }|q^i(z)|<1\mbox{ for all }0\leq i<j\}.\]
\end{definition}
$K_{\infty}$ is the set of points in $K(q)$ whose orbits by $q$ remain in $\mathbb D$ (the set $K_q$), the unit circle ($S_0$), and the parts of the iterated preimages of $S_0$ that remain in $\overline{\mathbb D}$ at each step (the sets $S_j$ for $j\geq1$).  See Figure \ref{FIG:KINF} for an example of $K(f_{n})$ with $q(z)=z^2-0.1+0.75i$ and several different values of $n$ compared to a sketch of $K_{\infty}$ for this polynomial $q$.

\begin{figure}[h]
\centering
\includegraphics[width=0.244\textwidth]{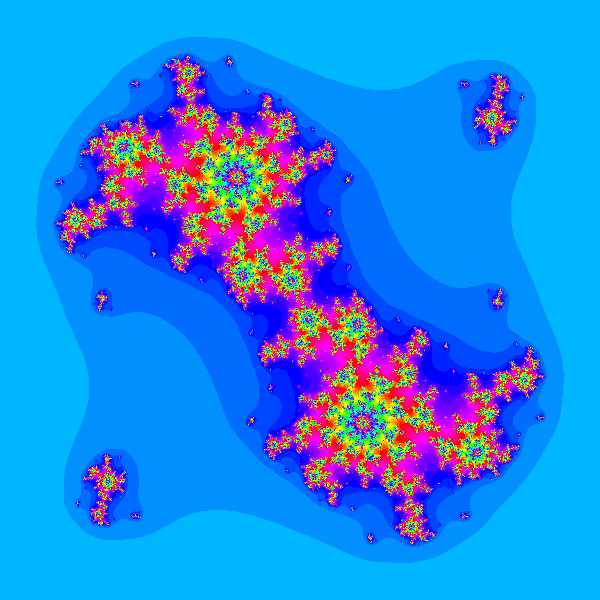}
\includegraphics[width=0.244\textwidth]{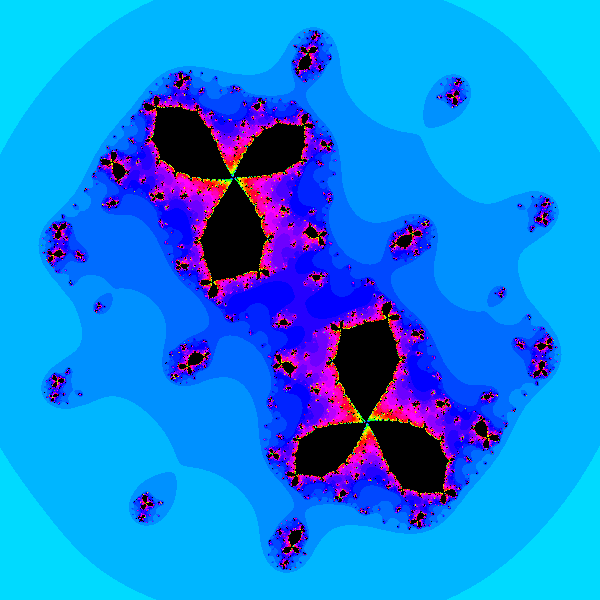}
\includegraphics[width=0.244\textwidth]{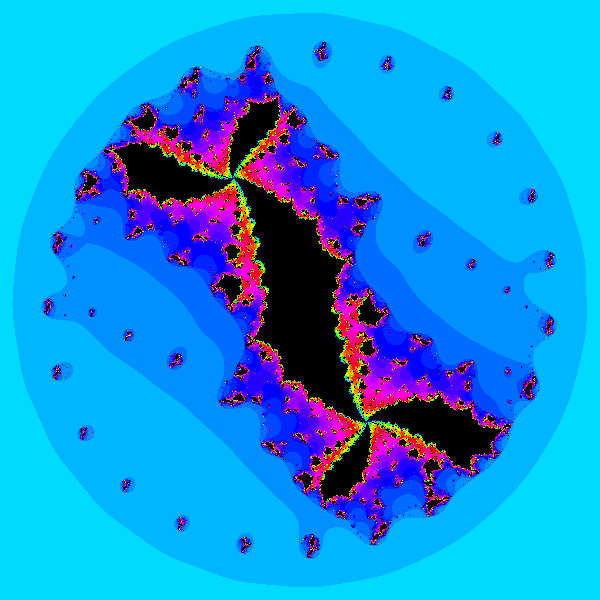}
\includegraphics[width=0.244\textwidth]{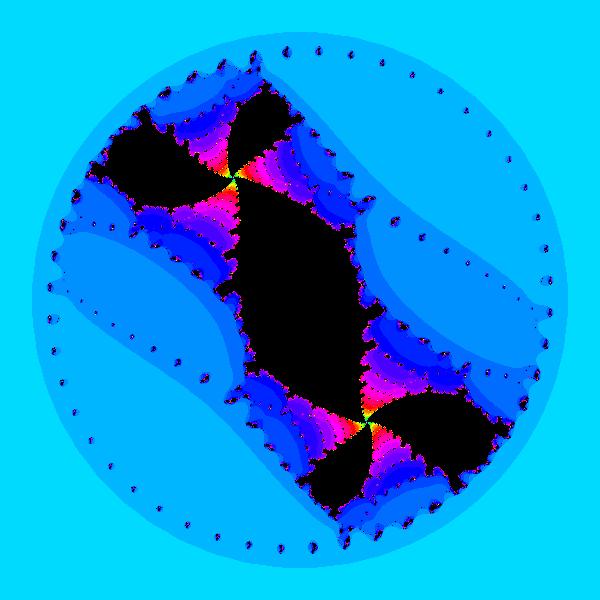}\\[6pt]
\includegraphics[width=0.48\textwidth]{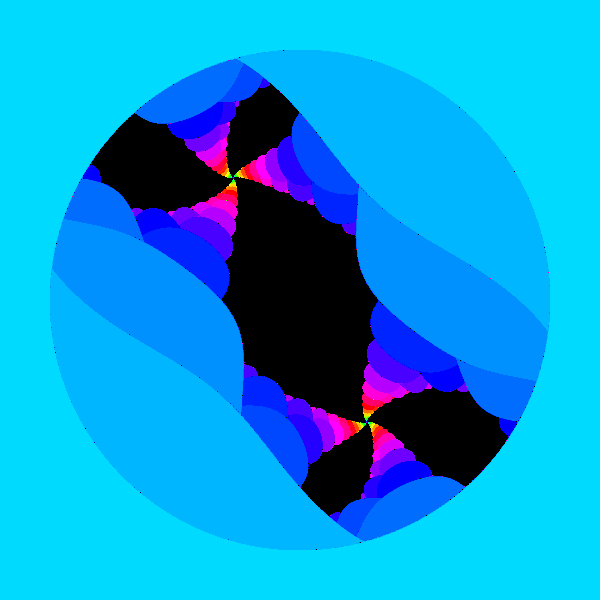}
\scalebox{0.389}{
\begin{picture}(0,0)%
\includegraphics{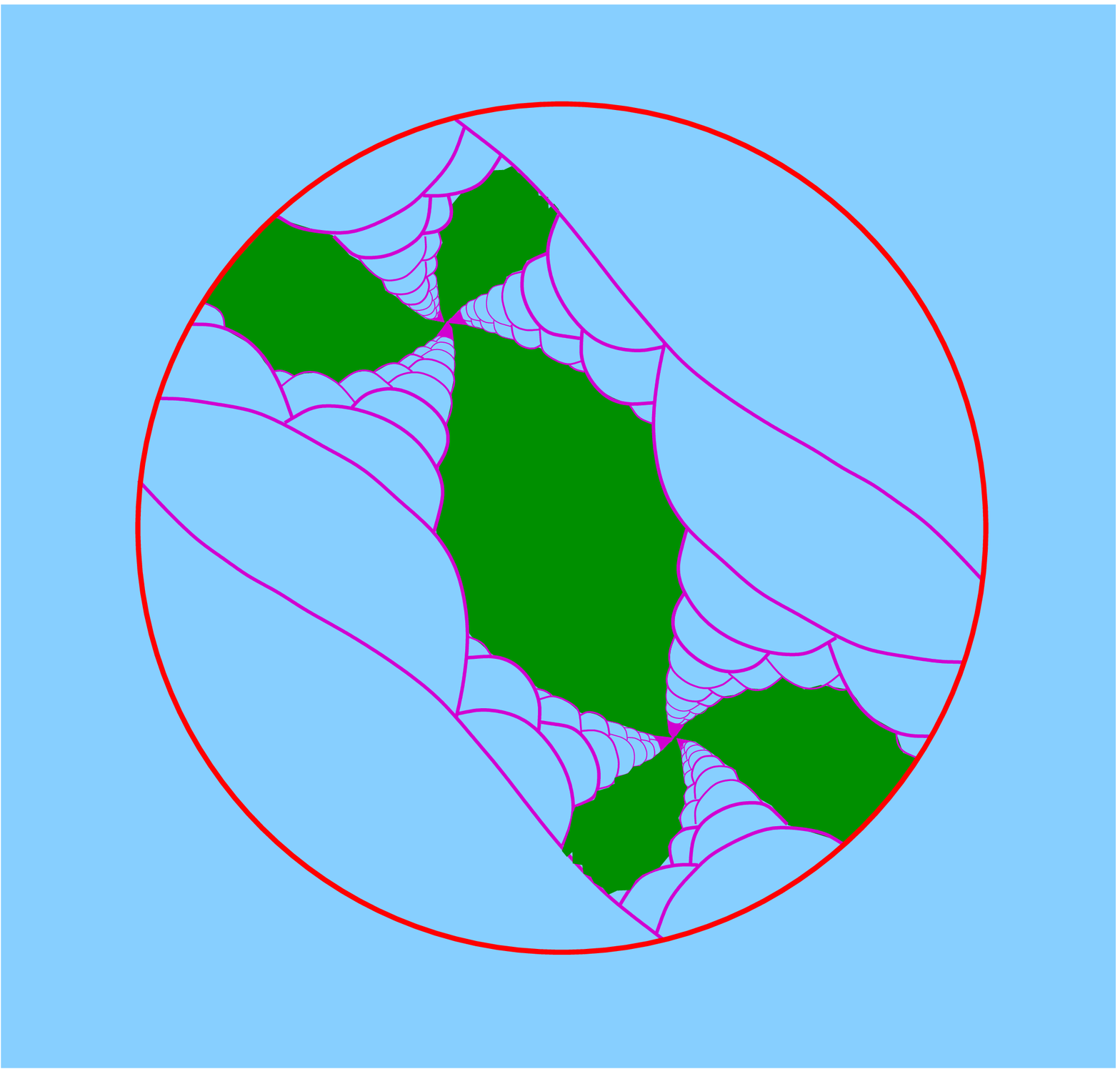}%
\end{picture}%
\setlength{\unitlength}{3947sp}%
\begingroup\makeatletter\ifx\SetFigFont\undefined%
\gdef\SetFigFont#1#2#3#4#5{%
  \reset@font\fontsize{#1}{#2pt}%
  \fontfamily{#3}\fontseries{#4}\fontshape{#5}%
  \selectfont}%
\fi\endgroup%
\begin{picture}(10092,9624)(1448,-10947)
\end{picture}%
}
\caption{Top, left to right: $K(f_{n})$ with $q(z)=z^2-0.1+0.75i$ and $n=6,12,25,50$. 
Bottom left: $K(f_{1800})$. 
Bottom right: Sketch of $K_{\infty}$, where $K_q$ is green, $S_0$ is red, and the sets $S_j$ are magenta.}
\label{FIG:KINF}
\end{figure}

By the construction of $K_{\infty}$, any point bounded a definite distance away from $K_{\infty}$ will eventually be mapped a definite distance outside the unit disk.  Then Lemma \ref{LEM:DISKBOUND} implies such a point must be in the basin of infinity for all $f_n$ with large enough $n$, so we have the following theorem.
\begin{theorem}
\label{THM:MAIN}
For any polynomial $q$,
\[K_{\infty}:=K_q\cup\bigcup_{j=0}^{\infty}S_j\supset\uplim_{n\rightarrow\infty}K(f_n)\supset\lowlim_{n\rightarrow\infty}K(f_n)\supset\partial K_q\cup\bigcup_{j=0}^{\infty}S_j.\]
\end{theorem}

What is happening here heuristically is that as long as the orbit of $z$ remains in $\mathbb D$, the polynomial $q(z)$ dominates the dynamics; if the orbit of $z$ leaves $\overline{\mathbb D}$, then the power map $z^n$ dominates. When the orbit hits $S_0$, it is not clear whether $q(z)$ or $z^n$ should win, so you get a point in the Julia set.  Now simple conditions that describe precisely when we can expect the closed unit disk, $\overline{\mathbb D}$, or the unit circle, $S_0$, as a limit follow from Theorem \ref{THM:MAIN}:
\begin{corollary}
\label{COR:MAIN}
Suppose $\deg q\geq2$ and $q$ has no fixed points in $S_0$.  Under the Hausdorff metric,
\begin{enumerate}
\item $\displaystyle\lim_{n\rightarrow\infty}K(f_{n})=\overline{\mathbb D}$ if and only if $q(\overline{\mathbb D})\subset\overline{\mathbb D}$, and
\item $\displaystyle\lim_{n\rightarrow\infty}K(f_{n})=S_0$ if and only if $q(\overline{\mathbb D})\cap\mathbb D=\emptyset$.
\end{enumerate}
\end{corollary}

With a couple additional assumptions on $q$, we also have the following stronger result.
\begin{theorem}
\label{THM:HYPERBOLIC}
If $\deg q\geq2$ and $q$ is hyperbolic with no attracting periodic points on $S_0$, then under the Hausdorff metric
\[\lim_{n\rightarrow\infty}K(f_{n})=K_{\infty}.\]
\end{theorem}

We are left with the following open question.
\begin{question}
In what ways can the hypotheses from Theorem \ref{THM:HYPERBOLIC} on polynomial $q$ be relaxed and still have $\lim_{n\rightarrow\infty} K(f_{n}) = K_\infty$?
\end{question}

Following a brief tour of background information and examples in Section \ref{SEC:BG}, we present the proof of Theorem \ref{THM:MAIN} in Section \ref{SEC:PROOF}.  Lastly, Section \ref{SEC:HYPERBOLIC} is devoted to the proofs of theorems that require specific hypotheses on $q$.

The authors are grateful to Roland Roeder at Indiana University Purdue University Indianapolis for his very helpful advice and the Butler University Mathematics Research Camp, where this project began. We are also very grateful to the referee for their insightful and helpful suggestions.  All images we created with the Dynamics Explorer \cite{DE} program.

\section{Background and Examples}\label{SEC:BG}

\subsection{Notation and Terminology}

The main results in this note rely on the convergence of sets in the Riemann sphere, $\hat{\mathbb C}$, where the convergence is with respect to the Hausdorff metric.  Given two sets $A,B$ in a metric space $(X,d)$, the Hausdorff distance $d_{\mathcal H}(A,B)$ between the sets is defined as
\begin{eqnarray*}
d_{\mathcal H}(A,B)&=&\max\left\{\sup_{a\in A}d(a,B),\sup_{b\in B}d(b,A)\right\}\\
&=&\max\left\{\sup_{a\in A}\inf_{b\in B}d(a,b),\sup_{b\in B}\inf_{a\in A}d(a,b)\right\}.
\end{eqnarray*}
The distance from each point in $A$ to $B$ has a least upper bound, and the same it true for each point from $B$ to $A$.  The Hausdorff distance is the supremum over all of these distances.  As an example, consider a regular hexagon $A$ with sides of length $r$ inscribed in a circle $B$ of radius $r$.  In this case, $d_{\mathcal H}(A,B)=r(1-\sqrt3/2)$, the shortest distance from the circle to the midpoint of any of the sides of the hexagon.  

Filled Julia sets $K(f_{n})$ are compact \cite{beardon}, bounded, and contained in the compact space $\hat{\mathbb C}$.  Moreover, with the Hausdorff metric $d_{\mathcal H}$, the space of all subsets of $\hat{\mathbb C}$ is complete \cite{henrikson}.  Suppose $S_n$ and $S$ are compact subsets of $\mathbb C$.  We say $S_n$ converges to $S$ and write $\lim_{n\rightarrow\infty}S_n=S$ if for all $\epsilon>0$, there is $N>0$ such that for all $n\geq N$, we have $d_{\mathcal H}(S_n, S)<\epsilon$.

We've also made use of Painlev\'e-Kuratowski set convergence \cite{ROCKAFELLAR}.  For a sequence of sets, $S_n$, we have
\begin{eqnarray*}
\lowlim_{n\rightarrow\infty}S_n&=&\{z\in\mathbb C\colon\uplim_{n\rightarrow\infty}d(z,S_n)=0\},\\
\uplim_{n\rightarrow\infty}S_n&=&\{z\in\mathbb C\colon\lowlim_{n\rightarrow\infty}d(z,S_n)=0\}.
\end{eqnarray*}
It follows immediately that $\lowlim_{n\rightarrow\infty}S_n\subset\uplim_{n\rightarrow\infty}S_n$.  We say $S_n$ converges to a set $S$ in the sense of Painlev\'e-Kuratowski if 
$\lowlim_{n\rightarrow\infty}S_n=\uplim_{n\rightarrow\infty}S_n=S$, 
or equivalently, $\uplim_{n\rightarrow\infty}S_n\subseteq\lowlim_{n\rightarrow\infty}S_n=S$.
It is shown in \cite{DONTCHEV} that for sequences of bounded sets, the notion of Painlev\'e-Kuratowski set convergence agrees with the notion of convergence with Hausdorff distance.

\subsection{Complex Dynamics}

We provide here only the fine details relevant to this paper. Thorough explorations of this subject and proof of all the facts below can be found in \cite{milnor,beardon,carleson}.  
The Fatou set of rational map $f\colon\hat{\mathbb C}\rightarrow\hat{\mathbb C}$, denoted $\mathcal F(f)$, is the set of points for which the iterates of $f$ form a normal family; the Julia set of $f$, denoted $J(f)$, is the complement of $\mathcal F(f)$ in $\hat{\mathbb C}$.  When $f$ is a polynomial map, the Julia set of $f$ is the boundary of the filled Julia set; that is, $J(f)=\partial K(f)$.

We say a point $z\in\hat{\mathbb C}$ is periodic for $f$ with period $k$ if $f^k(z)=z$ and the points $z,f(z),\dots,f^{k-1}(z)$ are all distinct.  The multiplier $\lambda$ of a periodic point $z_0$ of period $k$ is defined as
\[\lambda=(f^k)'(z_0)=\prod_{i=0}^{k-1}f'(f^i(z)).\]
If $|\lambda|<1$, then $z_0$ is attracting; if $\lambda>1$, $z_0$ is repelling; if $\lambda=1$, $z_0$ is indifferent.  Repelling periodic points are contained in $J(f)$; in fact, repelling periodic points are dense in $J(f)$.  Attracting periodic points, on the other hand, are contained in $\mathcal F(f)$.  Moreover, for every attracting periodic point $z_0$ of period $k$, there is an open neighborhood $B(z_0,\epsilon)$ such that $f^k(B(z_0,\epsilon))\subset B(z_0,\epsilon)$ and the orbit by $f^k$ of any point in $B(z_0,\epsilon)$ converges to $z_0$.  The set of all points whose orbits by $f^k$ converge to $z_0$ is called the basin of attraction for $z_0$.  Finally, a rational map is called hyperbolic if every point in $\mathcal F(f)$ converges to an attracting periodic cycle.

\subsection{Examples}\label{SEC:EX}

\begin{example}
Let $q(z)=0.75z^2+c_j$ with $c_1=0.21+0.017i$, $c_2=0.41+0.047i$, and $c_3=1.41+1.17i$.  For $c_1$, we have $q(\overline{\mathbb D})\subset\mathbb D$, so in this case, $K_{\infty}=\overline{\mathbb D}$.  For $c_3$, we have $q(\overline{\mathbb D})\cap\mathbb D=\emptyset$, so in this case, $K_{\infty}=S_0$.  Both of these cases follow from Corollary \ref{COR:MAIN}.  The more interesting case is $c_2$ in which $q(\mathbb D)\backslash\mathbb D\neq\emptyset$.  See Figure \ref{FIG:EX}.  The limit, should it exist, of $K(f_{n})$ is the set $K_{\infty}$, which is now significantly more complicated, neither the closed unit disk nor the unit circle.
\begin{figure}[h]
\centering
\hspace{0pt}\scalebox{0.4}{\includegraphics{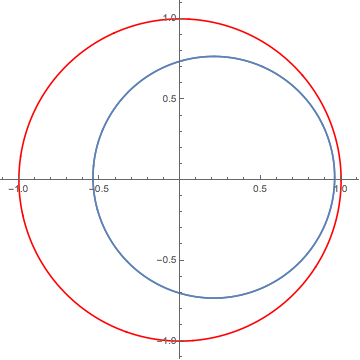}}
\scalebox{0.25}{\includegraphics{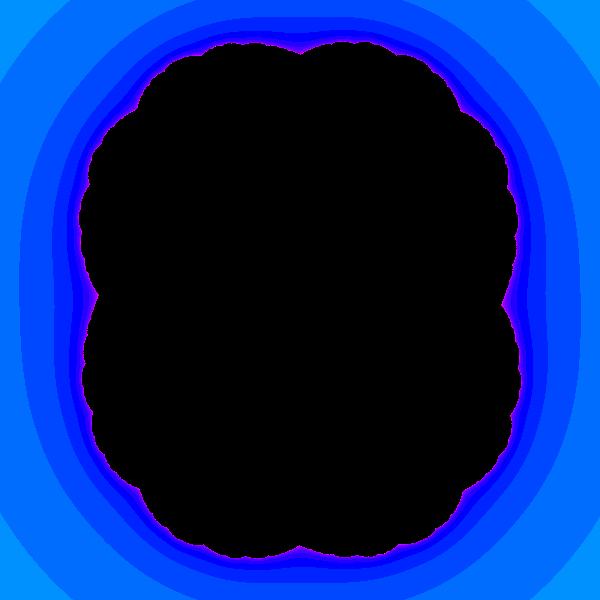}}
\scalebox{0.25}{\includegraphics{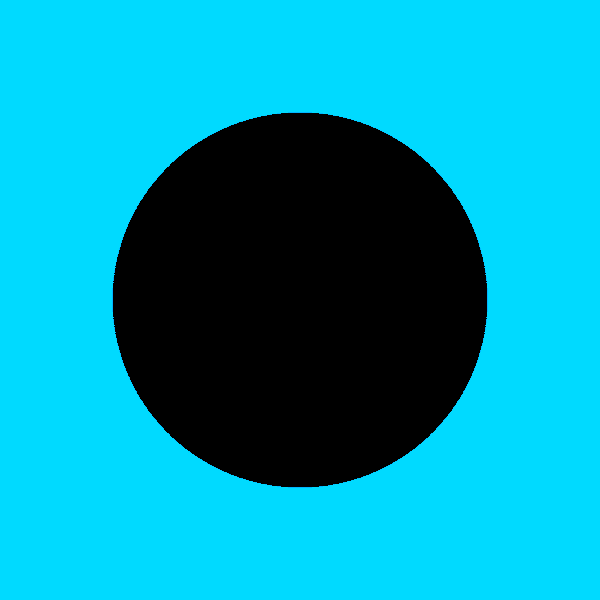}}\hspace{0pt}\\[6pt]
\hspace{0pt}\scalebox{0.4}{\includegraphics{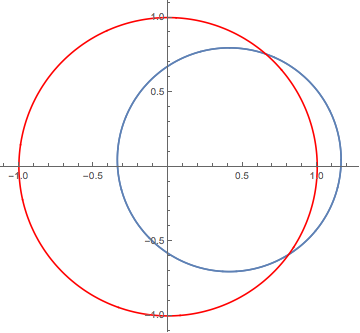}}
\scalebox{0.25}{\includegraphics{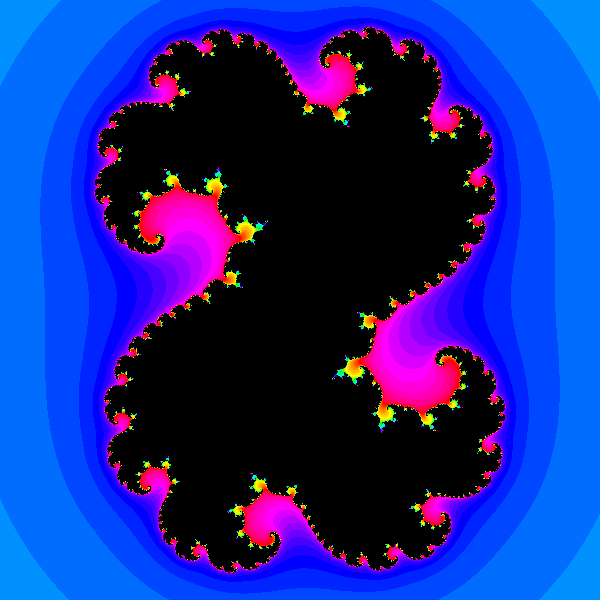}}
\scalebox{0.25}{\includegraphics{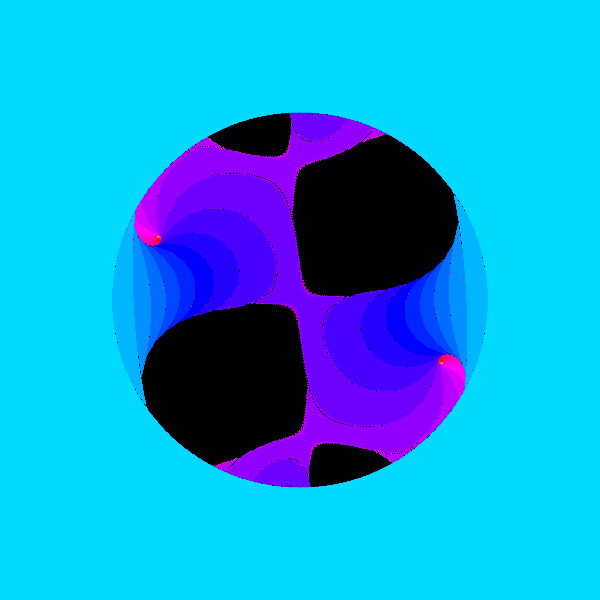}}\hspace{0pt}\\[6pt]
\hspace{0pt}\scalebox{0.4}{\includegraphics{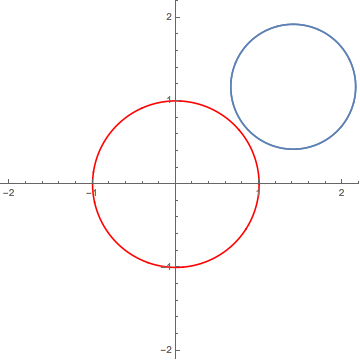}}
\scalebox{0.25}{\includegraphics{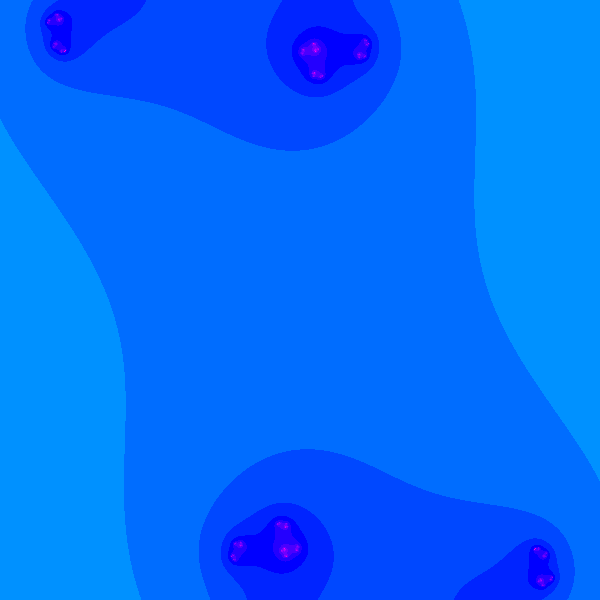}}
\scalebox{0.25}{\includegraphics{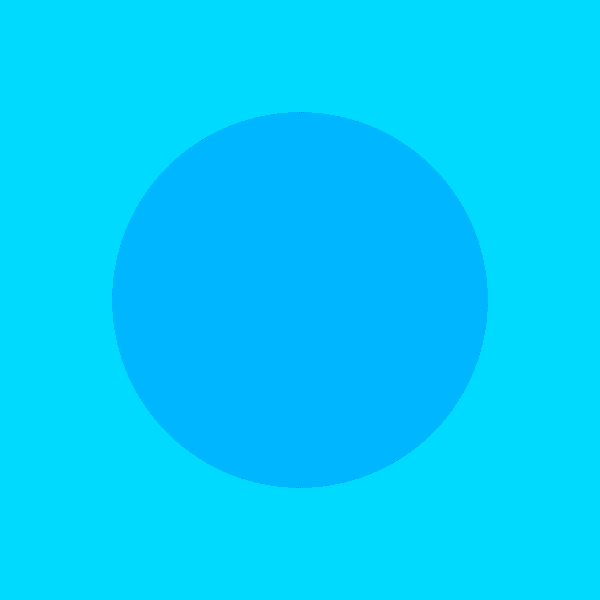}}\hspace{0pt}
\caption{Each row, left to right: $q(\overline{\mathbb D})$ for $q(z)=0.75z^2+c_j$, $K(q)$, and $K(f_{1800})$.  The first row is for $c_1=0.21+0.017i$, the second row is $c_2=0.41+0.047i$, and the third row is $c_3=1.41+1.17i$. Scale is the same in pictures in the last two columns.}
\label{FIG:EX}
\end{figure}
\end{example}

\section{Proof of Main Results}\label{SEC:PROOF}

\begin{proof}[Proof of Lemma \ref{LEM:DISKBOUND}]
Let $z\in \mathbb C\setminus\overline{\mathbb D}_{1+\epsilon}$.  We prove $|f^m_{n,q}(z)|\geq B^m$ for all $m\geq 1$ using induction.   Let $a_i$ be the coefficients of $q$, and pick $M>d(\max|a_i|)$.  Then for any $|z|>1$, 
we have $|q(z)|\leq M|z|^d$.  
Choose $B>\max{\{1,M\}}$ and $N>d$ large enough that $|z|^N>\max\{4B,2M|z|^d\}$. Let $n\geq N$.  Observe that
\begin{eqnarray*}
|f_{n,q}(z)|
\ \geq\ |z|^n-|q(z)|\ 
\geq\ |z|^n-M|z|^d
\ \geq\ |z|^n-\frac{1}{2}|z|^n\ \geq\ 2B\ >\ B.
\end{eqnarray*}

Now suppose for some $m\geq 1$, we know $|f_{n,q}^m(z)|\geq B^m$.  Let $z_m=f^m_{n,q}(z)$, and note that 
$|q(z_m)|\leq M|z_m|^d
<|z_m|^N$. Then for any $n\geq N$,
\begin{align*}
    \left|f^{m+1}_{n,q}(z)\right|\ 
     	    \geq\ \left|z_m^n\right|-\left|q(z_m)\right|\ 
             &\geq\ \left|z_m\right|^n-M\left|z_m\right|^d\\
            & \geq B^{mn}-B^{md}B\ 
             =\ B^{m+1}\left(B^{mn-m-1}-B^{md-m}\right)\ \geq\ B^{m+1}.
\end{align*}
It follows that $|f_{n,q}^m(z)|\geq B^m$ for all $m\geq 1$. Since $B>1$, the orbit of $z$ under $f_{n,q}$ escapes to infinity. Thus, $z\notin K(f_{n,q})$. 
\end{proof}

Before proving Theorems \ref{THM:HYPERBOLIC}, we need a couple more lemmas.

\begin{lemma}
\label{LEM:BOUNDED}
If $\{q^i(z_0)\}_{i=0}^{k-1}\subset\mathbb D$ and $|q^k(z_0)|=1$, then for any positive integer $m<k$, there is an $N$ such that for all $n\geq N$,
\[\{f_{n,q}^i(z_0)\}_{i=0}^m\subset\mathbb D.\]
Moreover, for all $\epsilon>0$ and any positive integer $m\leq k$, there is an $N$ such that for all $n\geq N$,
\[\max_{0\leq i\leq m}|f_{n,q}^i(z_0)-q^i(z_0)|<\epsilon.\]
\end{lemma}

\begin{proof}
This proof follows by continuity and is left to the reader.
\end{proof}

\begin{lemma}
\label{LEM:CIRCPREIM}
\[\lowlim_{n\rightarrow\infty}K(f_n)\supset\bigcup_{j=0}^{\infty}S_j\]
\end{lemma}

There is a body of work on the distribution of polynomial roots begun by Erd\"os and Tur\'an in \cite{erdos}.  Specific results in \cite{hughes,iz} dealing with the accumulation of polynomial roots around the unit circle could be applied to the polynomials $f_{n}(z)-z$ to find fixed points.  However, the case here is simpler because $n-d-1$ of the coefficients of $f_{n}$ are all zero.  Thus, we have a concise argument using the following potential theory lemma.
\begin{lemma}
\label{LEM:POT}
For any fixed degree $d$ nonzero polynomial, $q$, the zeros of the polynomial $f_{n,q}(z)=z^n+q(z)$ cluster uniformly around the unit circle as $n\rightarrow\infty$.  More specifically, for each $n$, let
\[\mu_n=\frac{1}{n}\sum_{f_{n,q}(z)=0}\delta_z,\]
where $\delta_z$ is a point mass at $z$, and the roots of $f_{n,q}$ are counted with multiplicity.  Then $\mu_n\rightarrow\mu$ weakly as $n\rightarrow\infty$, where $\mu$ is normalized Lebesgue measure on $S_0$.
\end{lemma}

\begin{proof}[Proof of Lemma \ref{LEM:POT}]
Note that
\[\mu=dd^c\log_{+}|z|,\qquad\mbox{where}\qquad\log_{+}|z|=\left\{\begin{array}{rl}
\log|z|,&\mbox{ if }|z|\geq1\\
0,&\mbox{ if }|z|<1.\end{array}\right.\]

Let $Z_q$ be the zero set of $q$, let $K$ be a compact subset of $\mathbb C\backslash(S_0\cup Z_q)$, and let $A$ be the maximum of $|q(z)|$ on $K$.  Then there is an $\epsilon>0$ such that for any $z\in K$, we have $|q(z)|>\epsilon$ and either $|z|\geq1+\epsilon$ or $|z|\leq1-\epsilon$. 

If $|z|\geq1+\epsilon$, then 
\begin{eqnarray}
\label{EQN:POT}
\frac{1}{n}\log|f_{n}(z)|&=&\frac{1}{n}\log\left|z^n\left(1+\frac{q(z)}{z^n}\right)\right|
\ =\ \frac{1}{n}\log|z^n|+\frac{1}{n}\log\left|1+\frac{q(z)}{z^n}\right|\nonumber
\end{eqnarray}
Using this 
and defining $C=\max\{\log(1+A/(1+\epsilon)),-\log(-1-A/(1+\epsilon))\}$, we have
\begin{eqnarray}
\label{EQN:POT3}
\log_{+}|z|-\frac{C}{n}&\leq&\frac{1}{n}\log|f_{n}(z)|\ \leq\ \log_{+}|z|+\frac{C}{n}.
\end{eqnarray}

If $|z|<1-\epsilon$, then there is an $N$ such that for all $n\geq N$, we have $z^n\leq\max\{\epsilon/2,A\}$.  Then
\begin{eqnarray*}
\frac{\epsilon}{2}\leq|q(z)|-|z^n|\ \leq\ |f_{n}(z)|&\leq&|z^n|+|q(z)|\leq2A.
\end{eqnarray*}
Noting that $\log_{+}|z|=0$ when $|z|\leq1-\epsilon$, we have for all $|z|\leq1+\epsilon$ that
\begin{eqnarray}
\label{EQN:POT2}
\log_{+}|z|+\frac{\log(\epsilon/2)}{n}\ \leq\ \frac{1}{n}\log\left(\frac{\epsilon}{2}\right)\ \leq\ \frac{1}{n}\log|f_{n}(z)|&\leq&
\frac{1}{n}\log(2A)\ =\ \log_{+}|z|+\frac{\log(2A)}{n}.
\end{eqnarray}
Using Equations (\ref{EQN:POT3}) and (\ref{EQN:POT2}), we have 
$\frac{1}{n}\log|f_{n}(z)|\rightarrow\log_{+}|z|$ 
uniformly on $K$ as $n\rightarrow\infty$; 
by the compactness theorem for families of subharmonic functions \cite[Theorem 4.1.9]{hormander}, 
it follows that $\frac{1}{n}\log|f_{n}(z)|\rightarrow\log_{+}|z|$ in $L^1_{loc}(\mathbb C)$.  Note that $dd^c\log_{+}|z|=\mu$, and we have from the Poincar\`e-Lelong formula \cite{gh} that $\frac{1}{n}dd^c\log|f_{n}(z)|=\mu_n$. 
Thus, we have
\[\mu_n=\frac{1}{n}dd^c\log|f_{n}(z)|\rightarrow dd^c\log_{+}|z|=\mu\]
weakly as $n\rightarrow\infty$.
\end{proof}

\begin{proof}[Proof of Lemma \ref{LEM:CIRCPREIM}]
We first deal only with the unit circle by showing that $S_0\subset\lowlim_{n\rightarrow\infty}K(f_{n})$.  Let $z\in S_0$ and $\epsilon>0$.  Define 
\[g_{n}(z):=f_{n}(z)-z,\]
so the zeros of $g_{n}$ are fixed points of $f_{n}$.  By Lemma \ref{LEM:POT}, the fixed points of $f_{n}$ cluster uniformly near the unit circle.  If any of the fixed points are repelling, then they are contained in $J(f_{n})$ \cite{milnor}.  Otherwise, they are attracting or indifferent, in which case they must be $\epsilon$ close to $J(f_{n})$ because $K(f_{n})\subset\mathbb D_{1+{\epsilon}}$.  It follows that $S_0\subset\lowlim_{n\rightarrow\infty}K(f_{n})$.

Since filled Julia sets are backward invariant, the preimages by $f_{n}$ of the fixed points clustering on $S_0$ will also be in $K(f_{n})$.  We now show that these preimages cluster on the sets $S_j$ (which are contained in the preimages of $S_0$).  Specifically, we must show that for large enough $n$, any point in $S_j$ for any $j$ is close to a preimage of one of the fixed points on $S_0$.  This task is made easier by the fact that by construction, 
the nonempty $S_j$ accumulate (when there are an infinite number of them) on the boundary of $K_q$.  To see this, define
\[\mathcal K_{J}=\bigcap_{i=0}^J\{z\in\mathbb C\colon|q^i(z)|<1\},\]
and note that for any $j>J$, we have from construction that $S_j\subset\mathcal K_{J}$ and 
$\lim_{J\rightarrow\infty}\mathcal K_{J}=K_q$. 
For any $J$, the set $\mathcal K_J\backslash K_q$ has compact closure, so 
for any $\epsilon>0$, there is a $J$ and points $\{z_1,\dots,z_{\ell})\subset\bigcup_{j<J}S_j$ such that 
\[\bigcup_{j=1}^{\infty}S_j\subset\bigcup_{i=1}^{\ell}B(z_{i},\epsilon/2).\]
That is, we have a finite open cover of the set of all the $S_j$ sets: for all $z\in\bigcup_{j=1}^{\infty}S_j$, there is an integer $1\leq i\leq\ell$ such that $|z-z_{i}|<\epsilon/2$.  


With this finite open cover, it is now straightforward to use Lemmas~\ref{LEM:BOUNDED} and \ref{LEM:POT} to show that for all $n$ sufficiently large, each $z_i$ (a center of one of the balls in the open cover) is close to some $w_{i,n}$, a preimage of a fixed point of $f_{n}$ on $S_0$.  Since each $w_{i,n}\in K(f_n)$, we can choose $n$ large enough so that for all $z_i$ we have $d(z_i,K(f_{n}))<\epsilon/2$.  Thus, for any $z\in S_j$ for any $j$, we have $d(z,K(f_{n}))<\epsilon$.
\end{proof}

It is intuitive by the construction of $K_{\infty}$ that points bound a definite distance away from $K_{\infty}$ will not be in $K(f_{n})$ for all sufficiently large $n$.  We nevertheless provide the formal statement and proof of this fact in the following lemma.

\begin{lemma}\label{LEM:HALFHAUS}
For any $\epsilon>0$, there is an $N$ such that for any $n\geq N$
\[d(z_0,K_{\infty})\geq\epsilon\mbox{\ \ implies\ \ }
z_0\notin K(f_{n,z}).\]
\end{lemma}
\begin{proof}
First we consider the case in which $|z_0|>1+\epsilon$.  By Lemma \ref{LEM:DISKBOUND} there is a large enough $N$ such that for all $n\geq N$, we have $z_0\notin K(f_{n,z})$. 
To attend to the remaining points, suppose $z_0\in\overline{\mathbb D}_{1+\epsilon}$ and  $d(z_0,K_{\infty})\geq \epsilon$.  However, since $S_0\subset K_{\infty}$, we only need to consider $z_0\in\overline{\mathbb D}_{1-\epsilon}$ with $d(z_0,K_{\infty})\geq \epsilon$.  Note that $\{z\in\mathbb C\colon d(z,K_{\infty})\geq\epsilon\}\cap\overline{\mathbb D}_{1-\epsilon}$ is a compact set, so there is some $j$ such that for all $z_0\in\{z\in\mathbb C\colon d(z,K_{\infty})\geq\epsilon\}\cap\overline{\mathbb D}_{1-\epsilon}$, we have $|q^j(z_0)|>1+\epsilon$.  By Lemma \ref{LEM:BOUNDED}, $|f^j_{n,q}(z_0)|>1+\epsilon$ for all sufficiently large $n$.  Again, by Lemma \ref{LEM:DISKBOUND}, we have $z_0\notin K(f_{n,z})$ in this case as well.
\end{proof}

\begin{lemma}\label{LEM:ATTRACT}
For any periodic orbit $\{z_i\}_{i=0}^{k-1}$ of $q$ contained in $\mathbb D$ and any $\epsilon>0$, there is an $N$ such that for all $n\geq N$, $f_{n}$ has a periodic orbit $\{z_{i,n}\}_{i=0}^{k-1}$ also contained $\mathbb D$ such that
\[\max_{0\leq i\leq k-1}|z_i-z_{i,n}|<\epsilon.\]
Moreover, if  $\{z_i\}_{i=0}^{k-1}$ is attracting (repelling) for $q$, then each cycle $\{z_{i,n}\}_{i=0}^{k-1}$ is attracting (repelling) for each corresponding $f_{n}$.
\end{lemma}

While zeros of non-contant polynomials depend continuously on the coefficients of the polynomial, the set of polynomials $\{f_{n}\}$ is discrete.  Nevertheless, Lemma \ref{LEM:ATTRACT} still follows quickly from Rouche's theorem and the fact that on any compact subset $K$ of $\mathbb D$, we have $f_{n}|_K\rightarrow q$ uniformly and $f'_{n,q}|_K\rightarrow q'$ uniformly, so we omit the proof.

%

\begin{proof}[Proof of Theorem \ref{THM:MAIN}]
From Lemma \ref{LEM:HALFHAUS}, it follows that for all $z$,  $d(z,K_{\infty})\leq\lowlim_{n\rightarrow\infty}d(z,K(f_{n}))$.  From this and Lemma \ref{LEM:CIRCPREIM}, we have
\[K_q\cup\bigcup_{j=0}^{\infty}S_j\supset\uplim_{n\rightarrow\infty}K(f_n)
\supset\lowlim_{n\rightarrow\infty}K(f_{n})\supset\bigcup_{j=0}^{\infty}S_j,\]
so it remains only to show that $\lowlim_{n\rightarrow\infty}K(f_{n})\supset\partial K_q$.  A point $z\in\partial K_q$ must be an accumulation point of $\bigcup_{j=0}^{\infty}S_j$ or in $\partial K(q)\cap\mathbb D=J(q)\cap\mathbb D$. In the former case, we have $z\in\lowlim_{n\rightarrow\infty}K(f_{n})$, so suppose the latter: $z\in J(q)\cap\mathbb D$, and let $\epsilon>0$.  Since repelling periodic points are dense in $J(q)$ \cite{milnor}, there is a repelling periodic point $z_0\in J(q)\cap\mathbb D$ such that $|z-z_0|<\epsilon/2$.  
By Lemma \ref{LEM:ATTRACT}, there is an $N$ such that for all $n\geq N$, there is a repelling periodic point $z_n\in J(f_{n})=\partial K(f_{n})$ such that $|z_0-z_n|<\epsilon/2$.  Thus, $d(z,K(f_{n}))\leq|z-z_n|<\epsilon$.
\end{proof}

\section{Hyperbolic and Other More Specific Maps}
\label{SEC:HYPERBOLIC}

We turn our attention now to the more specific cases in which $q$ is hyperbolic and has no periodic points on $S_0$.  

\begin{proof}[Proof of Theorem \ref{THM:HYPERBOLIC}]
By Theorem \ref{THM:MAIN}, it remains only to show that $K_q^{\circ}$, the interior of $K_q$, is a subset of $\lowlim_{n\rightarrow\infty}K(f_{n})$.  That is, we need only now show that any point of $K_q^{\circ}$ is close to $K(f_{n})$ for all sufficiently large $n$. 
If $K_q^{\circ}$ is empty, we are done, so we proceed with the assumption that $K_q^{\circ}$ is nonempty.

Since $q$ is hyperbolic, we know that the orbit of every point in $K_q^{\circ}\subset\mathcal F(q)\cap\mathbb D$ converges to an attracting periodic cycle for $q$.  Suppose $\{z_i\}_{i=0}^{k-1}$ is such an attracting cycle.  We will show that for large enough $n$, every $f_{n}$ has an attracting periodic near $\{z_i\}_{i=0}^{k-1}$.  Then every point in $K_q^{\circ}$ will have some forward image near an attracting cycle of $f_{n}$, ensuring these points are in $K(f_{n})$.

For each $z_i$ in the attracting cycle for $q$, there is a neighborhood $B(z_i,r_i)$ such that $q^k(B(z_i,r_i))\subset B(z_i,r_i)\subset\mathbb D$.  Let $r=\min_{i}r_i$. By Lemma \ref{LEM:ATTRACT}, there is an $N_0$ such that for all $n\geq N_0$, $f_{n}$ has an attracting periodic orbit $\{z_{i,n}\}_{i=0}^{k-1}$ where $z_{i,n}\in B(z_i,r)$. By Lemma \ref{LEM:BOUNDED}, there is an $N_1$ such that for all $n\geq N_1$ we also have $f_{n}^k(B(z_i,r))\subset B(z_i,r)$.  Thus, we have constructed a neighborhood that is forward invariant for all $f_{n}^k$ with $n\geq N_1$, so this neighborhood must be contained in each $K(f_{n})$.

We now show that compact subsets of $K_q^{\circ}$ are eventually mapped by $f_{n}$ into the forward invariant neighborhood we just constructed.  Let $\epsilon>0$ and $\mathcal K_{\epsilon}$ be a compact subset of $K_q^{\circ}$ such that for all $z\in\mathcal K_{\epsilon}$, there is some $z_0\in\partial K_q$ such that $d(z,z_0)<\epsilon$.  Then there is an $m$ such that 
\[q^m(\mathcal K_{\epsilon})\subset\bigcup_{i=0}^{k-1}B(z_i,r).\]
Again by Lemma \ref{LEM:BOUNDED}, there is an $N_2$ such that for all $m\geq N_2$ we also have
\[f_{n}^m(\mathcal K_{\epsilon})\subset\bigcup_{i=0}^{k-1}B(z_i,r).\]
Let $N=\max_{0\leq i\leq2}N_i$. For all $n\geq N$, it follows that the orbit by $f_{n}$ of any point $z\in\mathcal K_{\epsilon}$ converges to an attracting periodic cycle of $f_{n}$ contained in $\mathbb D$; that is, $\mathcal K_{\epsilon}\subset K(f_{n})$.  Then for any point $z\in K_q^{\circ}$, there is a $\hat z\in\mathcal K_{\epsilon}\subset K(f_{n})$ such that $d(z,\hat z)<\epsilon$.  
\end{proof}

\begin{proof}[Proof of Corollary \ref{COR:MAIN}]
We prove the first part of the corollary.  Suppose first that the image of $\overline{\mathbb D}$ under $q$ is contained in $\overline{\mathbb D}$ and let $0<\epsilon<1$.  By the open mapping theorem, we know $q(\mathbb D)$ is an open set in $\overline{\mathbb D}$, so $q(\mathbb D)\subset\mathbb D$.  Since $\deg q\geq2$, we have from the Denjoy-Wolff Theorem \cite{milnor} that $q$ has a fixed point $z_0$ in $\overline{\mathbb D}$ to which orbits of all points in any compact subset of $\mathbb D$ converge.  We have assumed $q$ has no fixed points in $S_0$, so $z_0\in\mathbb D$, and in this case, we also have from the Denjoy-Wolff Theorem that $z_0$ is the unique fixed point in $\mathbb D$.  From Lemma \ref{LEM:ATTRACT}, we have that for all $\epsilon>0$, there is an $N$ such that for all $n\geq N$, $f_{n}$ has an attracting fixed point $z_n$ with $|z_0-z_n|<\epsilon$ and no other fixed points in $\overline{\mathbb D}_{1-\epsilon}$. 
Thus, we have $\mathbb D_{1-\epsilon}\subset K(f_{n})$.  Combining this with Lemma \ref{LEM:DISKBOUND}, for any $\epsilon>0$, we may choose $N$ large enough such that
\[\mathbb D_{1-\epsilon}\subset K(f_{n})\subset\mathbb D_{1+\epsilon}.\]

Now suppose the image of $\overline{\mathbb D}$ under $q$ is not contained in $\overline{\mathbb D}$, so $q(\overline{\mathbb D})\backslash\overline{\mathbb D}$ is nonempty. 
Since $\overline{\mathbb D}$ is compact, $q(\overline{\mathbb D})$ is also compact, and since $\mathbb C\backslash\overline{\mathbb D}$ is open and $q(\overline{\mathbb D})\backslash\mathbb D$ is nonempty,   
there is some $z_0\in\mathbb D$ and $r>0$ such that $B(z_0,r)\subset\mathbb D$ and for any $z\in B(z_0,r)$, we have $|q(z)|>1$.  Then one can pick $N$ large enough that for any $n\geq N$ and any $z\in B(z_0,r)$, we have $|f_{n}(z)|>1$.  Then for all $n\geq N$, $z_0\notin K(f_n)$.  

For the second part of the theorem, assume $\overline{\mathbb D}$ under $q$ does not intersect $\mathbb D$, so $q(\overline{\mathbb D})\subset(\mathbb C\backslash\mathbb D)$.  That is,  for all $z_0\in\overline{\mathbb D}$, we have $|q(z_0)|\geq1$.  Let 
\[s=\min_{z\in\overline{\mathbb D}_{1-\epsilon}}\{|q(z)|\},\] 
so $1<s$.  Then $(s-1)/2>0$. Since $\overline{\mathbb D}_{1-\epsilon}$ is compact, we may choose this $N$ so that for any $z\in\mathbb D_{1-\epsilon}$,  we have $|z|^n<(s-1)/2$. 
Then for any $z\in\mathbb D_{1-\epsilon}$, we also have
\begin{align*}
    \left|f_{n}(z)\right|
    & \geq\big| |z|^n-|q(z)|\big|\\
    & =|q(z)|-|z|^n\\
    & > s-(s-1)/2\\
    & > 1+(s-1)/2.
\end{align*}
By Lemma \ref{LEM:DISKBOUND}, it follows that $\mathbb D_{1-\epsilon}$ is in the basin of infinity of $f_{n}$ for all $n\geq N$.  The result then follows from this fact and Lemma \ref{LEM:CIRCPREIM}.

Lastly, suppose the image $\overline{\mathbb D}$ under $q$ does intersect $\mathbb D$.  Then by Theorem \ref{THM:HYPERBOLIC}, if the limit exists, $\bigcup_{j>1}S_j\subset\mathbb D$ is nonempty, so the limit cannot be $S_0$.
\end{proof}

\bibliographystyle{plain}
\bibliography{main2}

\end{document}